\documentclass[11pt]{amsart}

\pdfoutput=1

\usepackage{cancel}
\usepackage{esint,amssymb} 
\usepackage{graphicx}
\usepackage{MnSymbol}
\usepackage{mathtools} 
\usepackage[colorlinks=true, pdfstartview=FitV, linkcolor=blue, citecolor=blue, urlcolor=blue,pagebackref=false]{hyperref}
\usepackage{microtype}

\usepackage{bm}
\usepackage{dsfont}
\usepackage{mathrsfs}
\usepackage{xcolor}

\newtheorem{proposition}{Proposition}
\newtheorem{theorem}{Theorem}
\newtheorem{lemma}[proposition]{Lemma}

\theoremstyle{remark}

\theoremstyle{definition}

\numberwithin{equation}{section}
\numberwithin{proposition}{section}
\numberwithin{figure}{section}
\numberwithin{table}{section}

\newcommand{\N}{\mathbb{N}}
\newcommand{\Q}{\mathbb{Q}}
\newcommand{\R}{\mathbb{R}}

\newcommand{\E}{\mathrm{E}}
\renewcommand{\P}{\mathrm{P}}
\newcommand{\EE}{\mathbf{E}}
\newcommand{\PP}{\mathbf{P}}
\newcommand{\F}{\mathcal{F}}

\renewcommand{\le}{\leqslant}
\renewcommand{\ge}{\geqslant}

\renewcommand{\bar}{\overline}

\renewcommand{\hat}{\widehat}

\newcommand{\Ll}{\left}
\newcommand{\Rr}{\right}
\renewcommand{\d}{\mathrm{d}}
\newcommand{\dr}{\partial}

\newcommand{\1}{\mathbf{1}}

\newcommand{\mcl}{\mathcal}
\newcommand{\msf}{\mathsf}

\newcommand{\msc}{\mathscr}
\newcommand{\al}{\alpha}
\newcommand{\be}{\beta}
\newcommand{\ga}{\gamma}

\newcommand{\si}{\sigma}

\newcommand{\prog}{\mathsf{Prog}}
\newcommand{\mart}{\mathsf{Mart}}

\newcommand{\bprog}{\mathbf{Prog}}
\newcommand{\bmart}{\mathbf{Mart}}

\begin{document}

\author{Jean-Christophe Mourrat}
\address[Jean-Christophe Mourrat]{Department of Mathematics, ENS Lyon and CNRS, Lyon, France}

\subjclass[2010]{}
\date{\today}

\title{Un-inverting the Parisi formula}

\begin{abstract}
The free energy of any system can be written as the supremum of a functional involving an energy term and an entropy term. Surprisingly, the limit free energy of mean-field spin glasses is expressed as an infimum instead, a phenomenon sometimes called an inverted variational principle. Using a stochastic-control representation of the Parisi functional and convex duality arguments, we rewrite this limit free energy as a supremum over martingales in a Wiener space. 
\end{abstract}

\maketitle

%
%
%
%
%
%
\section{Introduction and main result}

Let $h \in \R$ and $(\beta_p)_{p \ge 2}$ be a sequence of nonnegative real numbers such that the function $\xi(r) := \sum_{p \ge 2} \be_p^2 r^p$ is finite for every $r \in \R$. For every integer $N \ge 1$, let $(H_N'(\sigma))_{\sigma \in \R^N}$ be the centered Gaussian field with covariance such that, for every $\si, \tau \in \R^N$,
\begin{equation*}  
\mathbb E[H_N'(\sigma) H_N'(\tau)] = N \xi \Ll( \frac{\sigma \cdot \tau}{N} \Rr) ,
\end{equation*}
and for every $\si \in \R^N$, let 
\begin{equation*}  
H_N(\si) := H_N'(\sigma) + h \sum_{i = 1}^N \si_i.
\end{equation*}
The object of study of this paper is the limit free energy
\begin{equation}
\label{e.def.f}
f := \lim_{N\to\infty}\frac{1}{N}\mathbb E \log \bigg(\frac 1 {2^N}\sum_{\sigma\in\{-1,1\}^N} \exp(H_N(\sigma))\bigg).
\end{equation}
An explicit representation for this limit was conjectured in \cite{parisi1979infinite, parisi1980order, parisi1980sequence, parisi1983order, MPV} and then proved rigorously in \cite{gue03, Tpaper, pan.aom, pan}. In order to state this result, we denote by $\Pr([0,1])$ the set of probability measures over $[0,1]$. For each $\mu \in \Pr([0,1])$, we define $\Phi_\mu = \Phi_\mu(t,x) : [0,1]\times \R\to \R$ to be the solution to the backwards parabolic equation
\begin{equation}
\label{e.def.phimu}
-\dr_t \Phi_\mu(t,x) = \frac{\xi''(t)}{2} \Ll( \dr_x^2 \Phi_\mu(t,x) + \mu[0,t] \big(\dr_x \Phi_\mu(t,x)\big)^2 \Rr), 
\end{equation}
with terminal condition
\begin{equation}
\label{e.def.phimu.init}
\Phi_\mu(1,x) = \log \cosh x. 
\end{equation}
The \emph{Parisi formula} states that the limit free energy $f$  exists and is given by
\begin{equation}
\label{e.parisi.formula}
f = \inf_{\mu \in \Pr([0,1])} \Ll\{ \Phi_\mu(0,h)- \frac 1 2 \int_0^1 t \xi''(t) \mu[0,t] \, \d t \Rr\} .
\end{equation}
It was shown in \cite{auffinger2015parisi} that the variational problem in \eqref{e.parisi.formula} admits a unique minimizer. For ``most'' choices of $\xi$, this minimizer encodes the asymptotic law of the overlap between two independent samples from the Gibbs measure associated with the energy function $H_N$, see \cite[Corollaries~3.1 and 3.4]{pan}.

Any expression of the form $\log \E[\exp(f(X))]$ can be rewritten as a supremum over probability measures of an energy term and an entropy term, see for instance \cite[Corollaries~4.14 and 4.15]{boucheron2013concentration}. Such a formulation has a clear interpretation based on physical quantities or large-deviations considerations. It would seem reasonable that aspects of this structure be preserved in the limit of large system size, provided that we understand the asymptotics of the relevant energy and entropy terms. It thus initially came as a surprise that the expression in \eqref{e.parisi.formula} takes the form of an infimum instead, and to this day I am not aware of a compelling physical interpretation of this variational problem. In fact, for more complex models with multiple types, such a variational representation of the limit free energy is no longer valid in general, see \cite[Section~6]{mourrat2020nonconvex}.

The main goal of this paper is to give an alternative representation of the limit free energy $f$ as a supremum instead. In order to state the result, we define, for every $x \in \R$,
\begin{equation}
\label{e.def.small.phi}
\phi(x) := \log \cosh(x+h),
\end{equation}
and let $\phi^*$ denote the convex dual of $\phi$, so that for every $\lambda \in \R$,
\begin{equation*}  
\phi^*(\lambda) := \sup_{x \in \R} \Ll( \lambda x - \phi(x) \Rr) \in \R \cup \{+\infty\}.
\end{equation*}
A classical calculation yields that $\phi^*$ is infinite outside of $[-1,1]$, while for every $\lambda \in [-1,1]$, 
\begin{equation*}  
\phi^*(\lambda) = \frac 1 2 \Ll[ (1+\lambda) \log (1+\lambda) + (1-\lambda) \log(1-\lambda) \Rr] - \lambda h.
\end{equation*}

Let $\mathscr P = (\Omega,(\mcl F_t)_{t \in [0,1]}, \PP)$ be a filtered probability space. The $\sigma$-algebras $(\mcl F_t)_{t \in [0,1]}$ are assumed to be complete, that is, they contain every subset of any null-measure set. We also assume that $\msc P$ is sufficiently rich that one can define a Brownian motion $(W_t)_{t \in [0,1]}$  over it (in particular, the process~$W$ is adapted and has independent increments with respect to the filtration~$(\mcl F_t)_{t \in [0,1]}$). We denote by $\bmart$ the space of bounded martingales over $\msc P$. 
\begin{theorem}[un-inverted Parisi formula]
\label{t.main}
The limit free energy $f$ defined in~\eqref{e.def.f} is given by
\begin{multline}  
\label{e.main}
f = \sup_{\alpha \in \bmart} \bigg\{\EE \Ll[  \al_1 \int_0^1 \sqrt{\xi''(t)}\, \d W_t - \phi^*(\al_1)\Rr] 
\\- \frac 1 2 \sup_{t \in [0,1]} \int_t^1 \xi''(s) (s - \EE[\al_s^2 ]) \, \d s\bigg\}.
\end{multline}
Moreover, there exists a unique $\alpha \in \bmart$ that realizes the supremum in~\eqref{e.main}.
\end{theorem}
Theorem~\ref{t.main} is a consequence of a representation of the optimizers of \eqref{e.parisi.formula} and \eqref{e.main} as a dual pair in a saddle-point problem. For every $\mu \in \Pr([0,1])$  and $\al \in \bmart$, we define
\begin{multline}  
\label{e.def.Gamma}
\Gamma(\mu,\al) :=  \EE \bigg[ \al_1 \int_0^1 \sqrt{\xi''(t)}\, \d W_t - \phi^*(\al_1) 
\\ - \frac{1}{2}\int_{0}^{1}\xi''(t)\mu[0,t] (t - \alpha^{2}_t)\, \d t \bigg] .
\end{multline}
\begin{theorem}[saddle-point problem for optimizers]
\label{t.minimax}
Let $\bar \mu \in \Pr([0,1])$ denote the unique minimizer of \eqref{e.parisi.formula}, and let $\bar \alpha \in \bmart$ denote the unique maximizer of \eqref{e.main}. The limit free energy \eqref{e.def.f} is given by
\begin{equation}  
\label{e.f.supinf}
f = \Gamma(\bar \mu, \bar \alpha) = \inf_{\mu \in \Pr([0,1])} \, \sup_{\al \in \bmart} \Gamma(\mu,\alpha) = \sup_{\al \in \bmart} \, \inf_{\mu \in \Pr([0,1])}  \Gamma(\mu,\alpha).
\end{equation}
Moreover, the following two properties of $\bar \mu$ and $\bar \alpha$ are valid and, taken together, characterize the pair $(\bar \mu,\bar \al)$ uniquely in $\Pr([0,1]) \times \bmart$. 

(1) The support of $\bar \mu$ is a subset of the set of maximizers of the mapping 
\begin{equation}  
\label{e.def.map.t.min}
\Ll\{
\begin{array}{rcl}  
[0,1] & \to & \R \\
t & \mapsto & \int_t^1 \xi''(s) (s - \EE[\bar \al_s^2 ]) \, \d s .
\end{array}
\Rr.
\end{equation}

(2) Letting $(X_t)_{t \in [0,1]}$ be the strong solution to 
\begin{equation}  
\label{e.def.X.alphabar}
\Ll\{
\begin{array}{ll}  
X_0 = h, \\
\d X_t = \xi''(t) \bar \mu[0,t] \dr_x \Phi_{\bar \mu}(t, X_t) \, \d t + \sqrt{\xi''(t)} \, \d W_t,
\end{array}
\Rr.
\end{equation}
we have, for every $t \in [0,1]$,
\begin{equation}  
\label{e.def.bar.alpha}
\bar \al_t = \dr_x \Phi_{\bar \mu}(t,X_t) =  \dr_x \Phi_{\bar \mu}(0,h) + \int_0^t \sqrt{\xi''(s)} \dr_x^2 \Phi_{\bar \mu}(s,X_s) \, \d W_s.
\end{equation}
\end{theorem}

For each measure $\mu \in \Pr([0,1])$, the Parisi formula \eqref{e.parisi.formula} can be used to obtain an upper bound on the limit free energy $f$. In particular, we can write a first replica-symmetric upper bound obtained by imposing the support of $\mu$ to be a singleton, and then progress along the hierarchy of replica symmetry breaking by allowing the support of $\mu$ to contain two elements, then three elements, etc. Using the characterization above, we can associate to each $\mu \in \Pr([0,1])$ a corresponding lower bound on the free energy, as described in the next theorem. This result could for instance facilitate the construction of certified numerical approximations of the limit free energy.

\begin{theorem}[RSB lower bound]
\label{t.rsb} For every $\mu \in \Pr([0,1])$, we have
\begin{equation}
\label{e.rsb}
0 \le \Phi_\mu(0,h) -  \frac 1 2 \int_0^1 t \xi''(t) \mu[0,t] \, \d t  - f \le \frac 1 2 \int \bigg( g_\mu(t) - \inf_{\Ll[0,1\Rr]} g_\mu \bigg) \, \d \mu(t),
\end{equation}
where the function $g_\mu$ is such that, letting $(X_t)_{t \in [0,1]}$ be the strong solution to 
\begin{equation}  
\Ll\{
\begin{array}{ll}  
X_0 = h, \\
\d X_t = \xi''(t) \mu[0,t] \dr_x \Phi_{\mu}(t, X_t) \, \d t + \sqrt{\xi''(t)} \, \d W_t,
\end{array}
\Rr.
\end{equation} 
we have for every $t \in [0,1]$ that
\begin{equation}
\label{e.def.gmu}
g_\mu(t) := \int_t^1 \xi''(s) \Ll( \EE \Ll[ \Ll( \dr_x \Phi_\mu(s,X_s) \Rr) ^2 \Rr]  - s \Rr) 	\, \d s.
\end{equation}
Moreover, the right-hand side of \eqref{e.rsb} is zero if and only if $\mu \in \Pr([0,1])$ is the unique minimizer to \eqref{e.parisi.formula}. 
\end{theorem}

I expect that similar results can be proved for models with soft spins, using for instance \cite[Proposition~3.2]{mourrat2019parisi} and  \cite[Corollary~1.3]{mourrat2020extending}. The formulas for the limit free energy given in \cite{panchenko2005free} and \cite{mourrat2020extending} in this case are of a form similar to the infimum in \eqref{e.parisi.formula}, but with an additional parameter dependence that we then maximize over; this last step is meant to control the norm of the configuration $\si \in \R^N$ and is in line with classical large-deviations theory. A generalization of Theorem~\ref{t.main} to this case would allow us to write the limit free energy as a simple supremum instead. 

The present paper is inspired by a number of earlier works which we now briefly review. In \cite{jagannath2017low}, a dual problem was identified for the ground state energy of spherical models, and then extended to a similar representation for the free energy at any temperature in \cite{jagannath2018bounds}. In this context, the Parisi formula for the limit free energy can be rewritten in a form called the Crisanti-Sommers formula. This formula is simpler to manipulate than \eqref{e.parisi.formula}, and in particular, it is manifestly convex in the variable $\mu$. The dual problem identified in \cite{jagannath2017low, jagannath2018bounds} takes the form of an obstacle problem. 

Another series of inspiring works relates to the design of efficient algorithms that identify spin configurations whose energy is as high as possible. Important progress on this question was achieved in \cite{subag2021following} for spherical models. In this context, the energy value reached by the algorithm can be written explicitly as a simple function of $\xi$. The approach was then generalized to the models we consider here in \cite{montanari2021optimization, elalaoui2021optimization, sellke2021optimizing}, in the form of an approximate message passing algorithm with some tunable parameters. The asymptotic energy value reached by the algorithm can be described by an explicit formula, and optimizing over the free parameters gives us a variational formula for the best possible value $\mathsf{ALG}$ reached by this class of algorithms. This optimization problem has a structure similar to that in \eqref{e.main}. In particular, a key insight of \cite{subag2021following} is to consider algorithms that proceed via orthogonal increments, which in the limit yield martingales which we seek to optimize over as in \eqref{e.main}. The ground state energy can be represented as a Parisi-type formula analogous to \eqref{e.parisi.formula}, see \cite{auffinger2017parisi}, and it was shown in \cite{elalaoui2021optimization} that the original formulation of $\mathsf{ALG}$  as a supremum admits a dual representation that is similar to this Parisi formula for the ground state, the only difference being that the minimization is carried over a larger set. Closely related works include \cite{huang2021tight, huang2023algorithmic} where it is shown that a large class of algorithms will fail to find a configuration with energy exceeding $\mathsf{ALG}$. A method for sampling the Gibbs measure that may also be related to the present work was introduced in \cite{elalaoui2022sampling}. Other related works include~\cite{addario2019algorithmic, ho2023sampling, ho2022efficient}, where these optimization and sampling problems were investigated in detail for more analytically tractable models whose thermodynamics was derived in \cite{capocaccia1987existence, bovier2004derrida1, bovier2004derrida2}.

The starting point for the proofs of Theorem~\ref{t.main} and \ref{t.minimax} is a representation of the function $\Phi_\mu$ as the value function of a problem of stochastic control, which was introduced and developed in \cite{bovier2009aizenman, auffinger2015parisi, jagannath2016dynamic}. One of the main features of this representation is that it makes the convexity of the mapping $\mu \mapsto \Phi_\mu$ transparent, while I am not aware of any alternative method for verifying this convexity property. We then seek a dual to the minimization problem in~\eqref{e.parisi.formula}, in the sense of convex analysis, and the functional $\Gamma$ naturally appears. The interchange of the infimum and supremum in \eqref{e.f.supinf} requires some care, since for fixed $\mu$, the functional $\Gamma(\mu,\cdot)$ is not concave in general (although it is in the high-temperature regime $\xi''(1) \le 1$). We proceed by enlarging the probability space so as to ``randomize'' our choice of martingale, as in the construction of Nash equilibria with mixed strategies. Justifying that the supremum in~\eqref{e.main} is achieved requires yet another enlargement of the probability space. However, we can then identify this optimizer explicitly according to \eqref{e.def.bar.alpha}, and in particular observe that it is measurable with respect to the filtration of Brownian motion, and so we can ultimately conclude that these enlargements of the probability space were not necessary after all.

%
%
%
%
%
%
\section{Proofs}

We start by stating the representation of $\Phi_\mu$ alluded to above and due to~\cite{bovier2009aizenman, auffinger2015parisi, jagannath2016dynamic}. In order to do so, we let 
\begin{equation}
\label{e.def.Wiener}
\mathscr W := \big(C([0,1]), (\mcl F_t)_{t \in [0,1]}, \P\big)
\end{equation}
be the canonical Wiener space, with canonical random variable $(W_t)_{t \in [0,1]}$ and with associated expectation $\E$. The $\sigma$-algebras $(\mcl F_t)_{t \in [0,1]}$ are assumed to be complete. We denote by $\mart$ the space of bounded martingales over~$\mathscr W$, and by $\prog$ the space of bounded progressive processes over $\msc W$. 
By \cite{bovier2009aizenman, auffinger2015parisi, jagannath2016dynamic}, the quantity $\Phi_\mu(0,h)$ defined in \eqref{e.def.phimu}-\eqref{e.def.phimu.init} admits the variational representation
\begin{multline}
\label{e.phimu.var}
\Phi_\mu(0,h) = \sup_{\alpha\in\prog} \E \left[
\phi\left(\int_{0}^{1}\xi''(t)\mu[0,t]\alpha_t \, \d t+\int_0^1\sqrt{\xi''(t)}\, \d W_t\right)\right.\\
 - \Ll.\frac{1}{2}\int_{0}^{1}\xi''(t)\mu[0,t]\alpha^{2}_t\, \d t\right],
\end{multline}
and as a consequence, the Parisi formula \eqref{e.parisi.formula} can be rewritten as
\begin{multline}
\label{e.aufche}
f =  \inf_{\mu\in\Pr([0,1])}\sup_{\alpha\in\prog} \E \left[
\phi\left(\int_{0}^{1}\xi''(t)\mu[0,t]\alpha_t \, \d t+\int_0^1\sqrt{\xi''(t)}\, \d W_t\right)\right.\\
 - \Ll.\frac{1}{2}\int_{0}^{1}\xi''(t)\mu[0,t]\left(\alpha^{2}_t+t\right)\, \d t\right].
\end{multline}
The supremum in \eqref{e.phimu.var} can be interpreted as a stochastic control problem. Using this interpretation, it was shown in \cite{bovier2009aizenman, auffinger2015parisi, jagannath2016dynamic} that the supremum in~\eqref{e.phimu.var} is achieved, and an explicit description of the unique maximizer was given. It is relatively classical to verify that these results remain valid even if we enlarge the probability space and possibly allow for the control $\alpha$ to depend on additional randomness. Notice however that it is less classical to verify that the supremum on the right side of \eqref{e.main} does not depend on the choice of probability space, or to verify that the supremum is achieved, so it will be important for our purposes to be careful about such aspects. In the next lemma, we therefore restate the identity \eqref{e.phimu.var} in the generic probability space $\msc P$ (not necessarily the Wiener space $\msc W$), and describe precisely the identity of the maximizer. We denote by $\bprog$ the space of bounded progressive processes over the probability space $\msc P$.

\begin{lemma}[variational representation of $\Phi_\mu$ \cite{bovier2009aizenman, auffinger2015parisi, jagannath2016dynamic}]
\label{l.variat.phimu}
For every $\mu \in \Pr([0,1])$, we have
\begin{multline}
\label{e.variat.phimu}
\Phi_\mu(0,h) = \sup_{\alpha\in\bprog} \EE \left[
\phi\left(\int_{0}^{1}\xi''(t)\mu[0,t]\alpha_t \, \d t+\int_0^1\sqrt{\xi''(t)}\, \d W_t\right)\right.\\
 - \Ll.\frac{1}{2}\int_{0}^{1}\xi''(t)\mu[0,t]\alpha^{2}_t\, \d t\right].
\end{multline}
Moreover, let  $(X_t)_{t \in [0,1]}$ be the strong solution to 
\begin{equation}  
\label{e.AC.SDE}
\Ll\{
\begin{array}{ll}  
X_0 = h, \\
\d X_t = \xi''(t) \mu[0,t] \dr_x \Phi_{\mu}(t, X_t) \, \d t + \sqrt{\xi''(t)} \, \d W_t.
\end{array}
\Rr.
\end{equation}
A stochastic process $\al \in \bprog$ achieves the supremum in \eqref{e.variat.phimu} if and only if it satisfies, for almost every $t \in [0,1]$ with $\mu[0,t] > 0$, 
\begin{equation}  
\label{e.identity.alpha}
\al_t = \dr_x \Phi_{\mu}(t,X_t) = \dr_x \Phi_{\mu}(0,h) +  \int_0^t \sqrt{\xi''(s)} \dr_x^2 \Phi_{\mu}(s,X_s) \, \d W_s.
\end{equation}
\end{lemma}
\begin{proof}
We learn from \cite[Theorem~3]{auffinger2015parisi} or \cite[Lemma~18]{jagannath2016dynamic} that the representation~\eqref{e.variat.phimu} is valid if we replace $\bprog$ by the space of processes that are progressively measurable with respect to the $\sigma$-algebra generated by the Brownian motion $W$; or equivalently, if the probability space $\msc P$ is the canonical Wiener space $\msc W$. An examination of the proofs given there in fact also yields the identity \eqref{e.variat.phimu} in the more general context considered here. For the reader's convenience, we briefly justify this here as well. Since the supremum in \eqref{e.variat.phimu} is over a potentially larger set than the one considered in \cite{auffinger2015parisi, jagannath2016dynamic}, 
it suffices to show that for every $\al \in \bprog$, we have
\begin{multline}
\label{e.parisi.bcontrol}
\Phi_\mu(0,h) \ge   \EE \Ll[  \phi\left(\int_{0}^{1}\xi''(t)\mu[0,t]\alpha_t \, \d t 
    + \int_0^1\sqrt{\xi''(t)}\, \d W_t\right) \Rr.
    \\
    \Ll. - \frac{1}{2}\int_{0}^{1}\xi''(t)\mu[0,t] \alpha^{2}_t\, \d t
    \Rr] .
\end{multline}
We therefore fix $\al \in \bprog$ and for every $t \in [0,1]$, set
\begin{equation*}  
Y_t := h + \int_{0}^{t} \xi''(s)\mu[0,s]\alpha_s \, \d s 
    + \int_0^t \sqrt{\xi''(s)}\, \d W_s .
\end{equation*}
Using the regularity properties of $\Phi_\mu$ shown in \cite[Theorem~4]{jagannath2016dynamic}, we can apply It\^o's formula to get that 
\begin{multline*}  
\Phi_\mu(1,Y_1) = \Phi_\mu(0,h) + \int_0^1 \dr_t \Phi_\mu(t,Y_t) \, \d t + \int_0^1 \dr_x \Phi_\mu(t,Y_t) \, \d Y_t 
\\
+ \frac 1 2 \int_0^1 \dr_x^2 \Phi_\mu(t,Y_t) \, \d \langle Y \rangle_t.
\end{multline*}
Taking the expectation, we obtain that 
\begin{multline*}  
\EE \big[ \Phi_\mu(1,Y_1) \big] = \Phi_\mu(0,h) + \EE \int_0^1 \Big( \dr_t \Phi_\mu(t,Y_t) 
\\
 + \frac{\xi''(t)}{2}\Ll(2\mu[0,t] \al_t \dr_x \Phi_\mu (t,Y_t)   + \dr_x^2 \Phi_\mu(t,Y_t)\Rr)\Big) \, \d t.
\end{multline*}
By \eqref{e.def.phimu} and \cite[Theorem~1]{jagannath2016dynamic}, for almost every $t \in [0,1]$, we have for every $x \in \R$ that 
\begin{align}  
\notag
-\dr_t \Phi_\mu(t,x) 
& =  \frac{\xi''(t)}{2} \Ll( \dr_x^2 \Phi_\mu(t,x) + \mu[0,t] \big(\dr_x \Phi_\mu(t,x)\big)^2 \Rr) 
\\
\notag
& = \frac{\xi''(t)}{2}\sup_{a \in \R} \Ll( \dr_x^2 \Phi_\mu(t,x) + 2 \mu[0,t] a \dr_x \Phi_\mu(t,x) -   \mu[0,t]a^2  \Rr) 
\\
\label{e.ineq.control}
& \ge \frac{\xi''(t)}{2}\Ll( \dr_x^2 \Phi_\mu(t,x) + 2  \mu[0,t]\al_t \dr_x \Phi_\mu(t,x) -  \mu[0,t]\al_t^2   \Rr) .
\end{align}
Combining this with the previous display, we infer that 
\begin{equation*}  
\EE \big[ \Phi_\mu(1,Y_1) \big] \le \Phi_\mu(0,h) + \frac 1 2  \int_0^1\xi''(t) \mu[0,t] \EE[\al_t^2] \, \d t.
\end{equation*}
In view of \eqref{e.def.phimu.init} and \eqref{e.def.small.phi}, this is \eqref{e.parisi.bcontrol}, so the proof of \eqref{e.variat.phimu} is complete.

The second identity in \eqref{e.identity.alpha} is a consequence of It\^o's formula and the regularity properties of $\Phi_\mu$ given in \cite[Theorem~4]{jagannath2016dynamic}. The results of \cite[Theorem~3]{auffinger2015parisi} and \cite[Lemma~18]{jagannath2016dynamic} only state that \eqref{e.identity.alpha} is a sufficient condition for optimality, but again the proofs in fact yield that they are also necessary. This is also apparent from the argument given above, by an examination of the case of equality in \eqref{e.ineq.control}. 
\end{proof}

One key point of the representation in \eqref{e.variat.phimu} is that it is manifestly convex in $\mu$. In our next step, we rewrite this convex function as a supremum of affine functions.
\begin{lemma}[Supremum of affine functionals]
\label{l.prog.to.mart}
For every $\mu \in \Pr([0,1])$, we have
\begin{multline}  
\label{e.prog.to.mart}
\Phi_\mu(0,h) 
\\
= \sup_{\al \in \bmart} \EE \Ll[ \frac{1}{2}\int_{0}^{1}\xi''(t)\mu[0,t] \alpha^{2}_t\, \d t + \al_1 \int_0^1 \sqrt{\xi''(t)}\, \d W_t - \phi^*(\al_1)\Rr] .
\end{multline}
Moreover, there exists a unique maximizer to the variational problem in~\eqref{e.prog.to.mart}. Letting $(X_t)_{t \in [0,1]}$ be the solution to~\eqref{e.AC.SDE}, this unique maximizer $\al \in \bmart$ must satisfy~\eqref{e.identity.alpha} for every $t \in [0,1]$. 
\end{lemma}
\begin{proof}
We decompose the proof into two steps.

\medskip

\noindent \emph{Step 1}. Let $X \in L^2(\msc P)$. In this step, we show that
\begin{equation}
\label{e.measurable.selection}
\EE [\phi(X)] = \sup_{\lambda \in L^\infty(\msc P)} \EE \Ll[ \lambda X - \phi^*(\lambda) \Rr] ,
\end{equation}
and that the supremum in \eqref{e.measurable.selection} is achieved at $\lambda \in L^\infty(\msc P)$ if and only if $\lambda = \phi'(X)$. 

By the biconjugation theorem, we have that for every $x \in \R$,
\begin{equation}  
\label{e.biconj.phi}
\phi(x) = \sup_{\lambda \in \R} \Ll( \lambda x - \phi^*(\lambda) \Rr) .
\end{equation}
In particular, we have for every $x, \lambda \in \R$ that
\begin{equation}  
\label{e.ineq.convex.dual}
\phi(x) + \phi^*(\lambda) \ge \lambda x,
\end{equation}
and thus
\begin{equation*}  
\E [\phi(X)] \ge \sup_{\lambda \in L^\infty(\msc W)} \E \Ll[ \lambda X - \phi^*(\lambda) \Rr].
\end{equation*}
As a supremum of convex and lower semi-continuous functions, the function~$\phi^*$ is convex and lower semi-continuous. Since $\phi^*$ is in fact strictly convex on its effective domain $[-1,1]$, for each $x \in \R$, there exists a unique $\lambda \in \R$ such that 
\begin{equation}
\label{e.identity.convex.dual}
\phi(x) + \phi^*(\lambda) = \lambda x. 
\end{equation}
Since the derivative of $\phi^*$ diverges at the endpoints of the interval $[-1,1]$, this unique $\lambda$ must lie in the open interval $(-1,1)$. Similarly, for each $\lambda \in (-1,1)$, there exists a unique $x \in \R$ such that the identity \eqref{e.identity.convex.dual} is realized, and recall that we always have \eqref{e.ineq.convex.dual} in general. In short, the condition \eqref{e.identity.convex.dual} is equivalent to the statement that $\phi'(x) = \lambda$, and in particular, for every $x \in \R$,
\begin{equation*}  
\phi(x) = \phi'(x) x - \phi^*(\phi'(x)). 
\end{equation*}
This yields \eqref{e.measurable.selection} and the fact that the unique maximizer is $\lambda = \phi'(X)$. 

\medskip

\noindent \emph{Step 2.}
We let $(X_t)$ be the solution to~\eqref{e.AC.SDE}, and for every $t \in [0,1]$, we let
\begin{equation*}  
\hat \al_t := \dr_x \Phi_{\mu}(t,X_t) = \dr_x \Phi_{\mu}(0,h) +  \int_0^t \sqrt{\xi''(s)} \dr_x^2 \Phi_{\mu}(s,X_s) \, \d W_s.
\end{equation*}
Using \eqref{e.def.phimu.init} and \eqref{e.def.small.phi}, we notice that
\begin{equation}  
\label{e.observe.hatal1}
\hat \al_1 = \phi'(X_1 - h) = \phi' \Ll( \int_{0}^{1}\xi''(t)\mu[0,t]\hat \alpha_t \, \d t 
    + \int_0^1\sqrt{\xi''(t)}\, \d W_t \Rr).
\end{equation}
By the result of the previous step and Lemma~\ref{l.variat.phimu}, we can rewrite $\Phi_\mu(0,h)$ in the form of
\begin{multline}  
\label{e.expand}
 \sup_{\alpha\in\bprog} \sup_{\lambda \in {L^\infty(\mathscr P)}} \EE \left[
\lambda \left(\int_{0}^{1}\xi''(t)\mu[0,t]\alpha_t \, \d t 
    + \int_0^1\sqrt{\xi''(t)}\, \d W_t\right)
\right.
\\
\left.  - \phi^*(\lambda) -\frac{1}{2}\int_{0}^{1}\xi''(t)\mu[0,t] \alpha^{2}_t\, \d t\right] .
\end{multline}
Moreover, the choice of $\alpha = \hat \al$ and $\lambda = \hat \al_1$ realizes the supremum. Also, by Lemma~\ref{l.variat.phimu}, the random variable
\begin{equation*}  
 \int_{0}^{1}\xi''(t)\mu[0,t] \alpha_t \, \d t 
    + \int_0^1\sqrt{\xi''(t)}\, \d W_t
\end{equation*}
does not depend on the choice of optimizer $\al$ in \eqref{e.expand}. The result of the previous step and \eqref{e.observe.hatal1} therefore also yield that the choice of $\lambda = \hat \alpha_1$ is necessary for any maximizing pair $(\alpha, \lambda)$ of \eqref{e.expand}. If we also impose the maximizing pair to be such that $\alpha \in \bmart$ and $\al_1 = \lambda$, then by the martingale property this yields that $\al = \hat \al$. In the case when these additional constraints are imposed, the functional we optimize over in \eqref{e.expand} reduces to that in~\eqref{e.prog.to.mart}, so the proof is complete.
\end{proof}
By \eqref{e.parisi.formula} and Lemma~\ref{l.prog.to.mart}, this already shows that 
\begin{equation}  
\label{e.f.inf.sup}
f =  \inf_{\mu \in \Pr([0,1])} \, \sup_{\al \in \bmart} \Gamma(\mu,\alpha).
\end{equation}
Since the minimization problem over $\mu$ is in fact the same as that in \eqref{e.parisi.formula}, it is achieved at the same optimal $\mu$ and only there, by \cite{auffinger2015parisi}. Moreover, once $\mu$ is fixed, the maximization over $\al$ is achieved at the optimal $\al$ described by Lemma~\ref{l.prog.to.mart} and only there.  

Our next goal is to justify the interversion of the infimum and supremum in \eqref{e.f.inf.sup}.
For every bounded measurable $f : [0,1] \to \R$, we observe the integration by parts
\begin{multline}  
\label{e.ibp}
\int_0^1 f(s) \mu[0,s] \, \d s 
= \int_0^1 f(s)\int  \1_{\{t \le s\}} \, \d \mu(t) \, \d s 
\\
= \int \!\!\int_t^1 f(s) \, \d s \, \d \mu(t).
\end{multline}
It will be convenient to use this observation and rewrite \eqref{e.f.inf.sup} as
\begin{multline}  
\label{e.step0}
f = \inf_{\mu\in\Pr([0,1])} \sup_{\al \in \bmart} \EE \Ll[ \al_1 \int_0^1 \sqrt{\xi''(t)}\, \d W_t - \phi^*(\al_1) \Rr.
\\
+ \Ll. \frac{1}{2}\int \!\! \int_t^1  \xi''(s) (\alpha^{2}_s - s) \, \d s \, \d \mu(t)\Rr] .
\end{multline}
\begin{lemma}[Interchanging inf and sup]
\label{l.interchange}
Let
\begin{multline}  
\label{e.def.K0}
\msf K_0 := \bigg\{ \Ll(\EE \Ll[  \al_1 \int_0^1 \sqrt{\xi''(t)}\, \d W_t - \phi^*(\al_1)\Rr] ,(\EE[\al_t^2])_{t \in [0,1]} \Rr)\ \big\mid 
\\  \al \in \bmart, \ \EE[\phi^*(\al_1)] < +\infty \bigg\},
\end{multline}
and let $\msf K$ be the closure of the convex hull of $\msf K_0$ with respect to the topology of the space $\R \times L^1([0,1])$. 
The limit free energy \eqref{e.def.f} is given by
\begin{equation}  
\label{e.interchange}
f = \sup_{(\chi, \gamma) \in \msf K} \bigg\{\chi+ \frac 1 2 \inf_{t \in [0,1]} \int_t^1 \xi''(s) (\ga_s - s) \, \d s\bigg\}.
\end{equation}
\end{lemma}

\begin{proof}[Proof of Lemma~\ref{l.interchange}]
We can rewrite \eqref{e.step0} in the form of
\begin{equation}
\label{e.step1}
f = \inf_{\mu\in\Pr([0,1])} \sup_{(\chi,\ga) \in \msf K_0} \Ll( \chi + \frac{1}{2}\int \!\! \int_t^1  \xi''(s) (\ga_s - s) \, \d s \, \d \mu(t) \Rr).
\end{equation}
Let us call $G(\mu,\chi,\ga)$ the functional between the large parentheses in \eqref{e.step1}. For each $\mu$, the functional $G(\mu,\cdot, \cdot)$ is affine, so the supremum in \eqref{e.step1} is not changed if we replace $\msf K_0$ by its convex hull. For every $\ga, \rho \in L^1([0,1])$,  
\begin{multline}  
\label{e.continuity.G}
\sup_{t \in [0,1]} \Ll| \int_t^1  \xi''(s) (\ga_s - s) \, \d s - \int_t^1  \xi''(s) (\rho_s - s) \, \d s \Rr| 
\\
\le \xi''(1) \int_0^1 |\ga_s - \rho_s| \, \d s.
\end{multline}
The functional $G(\mu,\cdot,\cdot)$ is therefore continuous for the topology of $\R \times L^1([0,1])$. We deduce that the supremum in \eqref{e.step1} is not changed if we replace $\msf K_0$ by $\msf K$, that is,
\begin{equation}
\label{e.step2}
f = \inf_{\mu\in\Pr([0,1])} \sup_{(\chi,\ga) \in \msf K} \Ll( \chi + \frac{1}{2}\int \!\! \int_t^1  \xi''(s) (\ga_s - s) \, \d s \, \d \mu(t) \Rr).
\end{equation}

We now argue that the set $\msf K_0$ is precompact. First, for every $\al \in \mart$, the condition $\E[\phi^*(\al_1)] < +\infty$ is equivalent to the condition that $\al_1$ takes values in $[-1,1]$ almost surely. As a result, if  $(\chi^{(n)}, \ga^{(n)})$ denotes a sequence of elements of $\msf K_0$, we have that $\chi^{(n)}$ and $\sup_{t \in [0,1]} |\ga^{(n)}(t)|$ are bounded uniformly over $n$. We can therefore find a subsequence along which $\chi^{(n)}$ and $\ga^{(n)}_t$ converge for each $t \in \Q \cap [0,1]$. Using also that for each $n$, the mapping $t \mapsto\ga^{(n)}_t$ is non-decreasing, we deduce that $(\chi^{(n)}, \ga^{(n)})$ converges in $\R \times L^1([0,1])$ along the subsequence. This shows that $\msf K_0$ is precompact, and therefore that its closed convex hull $\msf K$ is compact, by \cite[Theorem~5.35]{aliprantis2006infinite}.

Endowing the space of probability measures $\Pr([0,1])$ with the topology of weak convergence turns this space into a compact set. Using again~\eqref{e.continuity.G}, we see that $G$ is jointly continuous. We have already used that for each fixed~$\mu \in \Pr([0,1])$, the mapping~$G(\mu,\cdot,\cdot)$ is affine; and for each fixed $(\chi,\ga) \in \R \times L^1([0,1])$, the mapping $G(\cdot,\chi,\ga)$ is affine. By the minimax theorem~\cite{fan1953minimax, sion1958minimax}, we can therefore exchange the infimum and the supremum in~\eqref{e.step2} and obtain that 
\begin{equation*}  
f = \sup_{(\chi,\ga) \in \msf K} \inf_{\mu\in\Pr([0,1])} \Ll( \chi + \frac{1}{2}\int \!\! \int_t^1  \xi''(s) (\ga_s - s) \, \d s \, \d \mu(t) \Rr).
\end{equation*}
The infimum is achieved for measures $\mu$ that are supported on minimizers of the mapping 
\begin{equation*}  
t \mapsto \int_t^1  \xi''(s) (\ga_s - s) \, \d s.
\end{equation*}
This yields the announced result.
\end{proof}

While Lemma~\ref{l.interchange} does indeed perform some interchange of infimum and supremum, the expression of the optimization as a supremum over the set $\msf K$ is not very satisfactory. The goal of the next lemma is to revert this back to an optimization problem over the space of martingales. The proof of this lemma will be simplified if we assume that the $\sigma$-algebra $\F_0$ of the probability space~$\msc P$ is sufficiently rich, so that we can perform a ``randomization'' operation on the space of martingales, as in the construction of Nash equilibria. We could also show the next lemma directly without this assumption, by showing that the ``randomized'' martingales we use can be approximated arbitrarily closely by martingales that are measurable with respect to the filtration generated by the Brownian motion~$W$. Since we will ultimately show that the supremum is in fact achieved at a martingale that is measurable with respect to the filtration generated by $W$, there is ultimately no loss in making this additional assumption here.
\begin{lemma}
\label{l.back.to.mart}
Suppose that the $\sigma$-algebra $\F_0$ of the probability space $\msc P$ is sufficiently rich that there exists an $\F_0$-measurable random variable that is uniformly distributed over $[0,1]$. Then the identity \eqref{e.main} is valid.
\end{lemma}
\begin{proof}
Lemma~\ref{l.interchange} clearly implies that 
\begin{multline}  
\label{e.first.bound}
f \ge  \sup_{\alpha \in \bmart} \bigg\{\EE \Ll[  \al_1 \int_0^1 \sqrt{\xi''(t)}\, \d W_t - \phi^*(\al_1)\Rr] 
\\- \frac 1 2 \sup_{t \in [0,1]} \int_t^1 \xi''(s) (s - \EE[\al_s^2 ]) \, \d s\bigg\},
\end{multline}
so we only need to show the converse bound. Notice carefully that the definition of $\msf K_0$ in \eqref{e.def.K0} depends on the identity of the probability space~$\msc P$. For the bound converse to \eqref{e.first.bound}, we will appeal to Lemma~\ref{l.interchange} applied to the case when the underlying probability space is the canonical Wiener space~$\msc W$. We recall that we denote by $\E$ the expectation over $\msc W$, and by $\mart$ the space of bounded martingales over $\msc W$, while we denote by $\bmart$ the space of bounded martingales over $\msc P$. There is a canonical injection from $\mart$ to $\bmart$ given by $\alpha \mapsto \alpha \circ W$; but the space $\bmart$ may be larger than $\mart$ in general. We let $\msf K_0$ be as in \eqref{e.def.K0} but with the underlying probability space being $\msc W$, we let $\msf K_1$ be the convex hull of $\msf K_0$, and $\msf K$ be the closure of~$\msf K_1$. By \eqref{e.continuity.G}, the representation in \eqref{e.interchange} is still valid if we replace $\msf K$ by~$\msf K_1$. Moreover, every element of $(\chi,\ga) \in \msf K_1$ can be represented in the form 
\begin{equation*}  
(\chi,\ga) =\sum_{i = 1}^n c_i \Ll(\E \Ll[  \al^{(i)}_1 \int_0^1 \sqrt{\xi''(t)}\, \d W_t - \phi^*(\al_1^{(i)})\Rr] ,\big(\E\big[(\al^{(i)}_t)^2\big]\big)_{t \in [0,1]} \Rr),
\end{equation*}
were $c_1,\ldots, c_n \in [0,1]$ are such that $\sum_i c_i = 1$ and $\al^{(1)}, \ldots, \al^{(n)} \in \mart$ take values in $[-1,1]$. We identify the martingales $\al^{(1)}, \ldots, \al^{(n)}$ with elements of $\bmart$ through the injection $\al \mapsto \al \circ W$ mentioned above. Notice that by construction, the martingales $\al^{(1)}, \ldots, \al^{(n)}$ are independent of $\F_0$. 
Under our assumption on $\msc P$, there exists an $\F_0$-measurable random variable $N$ such that for every $i \in \{1,\ldots, n\}$,
\begin{equation*}  
\P \Ll[ N = i \Rr] = c_i.
\end{equation*}
For every $t \in [0,1]$, we set
\begin{equation*}  
\beta_t := \alpha^{(N)}_t.
\end{equation*}
Since $N$ is $\F_0$-measurable and the martingales $\al^{(1)}, \ldots, \al^{(n)}$ are independent of $\F_0$, we see that $\be \in \bmart$, and a direct calculation gives that 
\begin{equation*}  
\Ll(\EE \Ll[  \beta_1 \int_0^1 \sqrt{\xi''(t)}\, \d W_t - \phi^*(\beta_1)\Rr] ,\big(\EE\big[(\be_t)^2\big]\big)_{t \in [0,1]} \Rr) = (\chi,\ga).
\end{equation*}
So we have shown that every pair $(\chi,\ga) \in \msf K_1$ can be represented in the form above. Since we have also observed that 
\begin{equation*}  
f = \sup_{(\chi, \gamma) \in \msf K_1} \bigg\{\chi+ \frac 1 2 \inf_{t \in [0,1]} \int_t^1 \xi''(s) (\ga_s - s) \, \d s\bigg\},
\end{equation*}
and recalling also \eqref{e.first.bound}, we obtain the result.
\end{proof}
We continue by showing the existence of a martingale that achieves the supremum in \eqref{e.main}, at the cost of possibly changing the probability space again.

\begin{lemma}[Existence of maximizing martingale]
\label{l.max.mart}
There exists a probability space $\msc P$ such that the identity \eqref{e.main} is valid and the supremum appearing there is achieved.
\end{lemma}
\begin{proof}
Let $\msc P$ be a probability space satisfying the assumption of Lemma~\ref{l.back.to.mart}, and let $(\al^{(n)})_{n \in \N}$ be a sequence of elements of $\bmart$ such that 
\begin{multline*}  
f = \lim_{n \to \infty} 
\bigg\{\EE \Ll[  \al^{(n)}_1 \int_0^1 \sqrt{\xi''(t)}\, \d W_t - \phi^*\big(\al^{(n)}_1\big)\Rr] 
\\+ \frac 1 2 \inf_{t \in [0,1]} \int_t^1 \xi''(s) \big(\EE\big[(\al^{(n)}_s)^2 \big] - s\big) \, \d s\bigg\}.
\end{multline*}
For $n$ sufficiently large, the martingale $\al^{(n)}$ must take values in $[-1,1]$. In particular, for each $t \in [0,1]$, the family of random variables $(\al^{(n)}_t)_{n \in \N}$ is tight. By Prokhorov's theorem, we can therefore find a subsequence $(k_n)_{n \in \N}$, a probability space $\bar {\msc P} = (\bar{\Omega},\bar{\F},\bar{\PP})$ and random variables $(W,(\al_t)_{t \in \Q \cap [0,1]})$ over $\bar{\msc P}$ such that for every integer $\ell \ge 1$, $t_1, \ldots, t_\ell \in \Q \cap [0,1]$, and bounded continuous function $G : C([0,1]) \times \R^\ell \to \R$, 
\begin{equation}  
\label{e.conv.in.law}
\lim_{n \to \infty} \EE \Ll[ G(W,\al^{(k_n)}_{t_1}, \ldots, \al^{(k_n)}_{t_\ell}) \Rr] = \bar{\EE} \Ll[ G(W,\al_{t_1}, \ldots, \al_{t_\ell}) \Rr].
\end{equation}
Without loss of generality we assume from now on that the convergence above is valid along the full sequence, that is, we can take $k_n = n$. 
For every $t \in [0,1]$, we let $\bar{\F}_t$ be the $\sigma$-algebra on $\bar{\msc P}$ generated by the random variables $(W_s)_{s \le t}$ and $(\al_s)_{s \in \Q \cap [0,t]}$. The process~$W$ is a Brownian motion with respect to this filtration, and $(\al_t)_{t \in \Q \cap [0,1]}$ is a martingale with respect to $(\bar{\F}_t)_{t \in \Q \cap [0,1]}$. We extend $\al$ by setting, for every $t \in [0,1]$,
\begin{equation*}  
\al_t := \bar{\EE} \Ll[ \al_1 \mid \bar\F_t \Rr] .
\end{equation*}
This turns $\al$ into a martingale with respect to the filtration $(\bar \F_t)_{t \in [0,1]}$. Since~$\al^{(n)}_1$ takes values in $[-1,1]$ and $\phi^*$ is continuous over this interval, we have that 
\begin{equation*}  
\lim_{n \to \infty} \EE \big[\phi^*\big(\al^{(n)}_1\big)\big] = \bar\EE \Ll[\phi^*(\al_1)\Rr]  .
\end{equation*}
In order to control the continuity of the linear mapping $W \mapsto \int_0^1 \sqrt{\xi''(t)}\, \d W_t$, we introduce the notation
\begin{equation*}  
\zeta(t) :=\sqrt{\xi''(t)}. 
\end{equation*}
We observe that $\zeta$ is continuously differentiable on $[0,1]$ whenever $\xi''(0) \neq 0$ or $\xi''(0) = \xi'''(0) = 0$, while in all cases, we can find a constant $C < \infty$ such that for every $t \in (0,1]$,
\begin{equation*}  
\zeta'(t) \le C t^{-1/2}.
\end{equation*}
Integrating by parts (or applying It\^o's formula), we have that
\begin{align*}  
\zeta(1) W_1 = \int_0^1 \zeta'(t) W_t \, \d t + \int_0^1 \zeta(t) \, \d W_t,
\end{align*}
and in particular, for a possibly larger value of the constant $C < \infty$, we have that for every $W \in C([0,1])$,
\begin{equation*}  
\Ll| \int_0^1 \zeta(t) \, \d W_t \Rr| \le C \|W\|_{L^\infty([0,1])}.
\end{equation*}
From this, \eqref{e.conv.in.law}, the fact that $\al^{(n)}_1$ takes values in $[-1,1]$, and an approximation argument, we deduce that 
\begin{equation*}  
\lim_{n \to \infty} \EE \Ll[  \al^{(n)}_1 \int_0^1 \sqrt{\xi''(t)}\, \d W_t \Rr] = \bar \EE \Ll[  \al_1 \int_0^1 \sqrt{\xi''(t)}\, \d W_t \Rr].
\end{equation*}
It also follows from \eqref{e.conv.in.law} that for every $t \in \Q \cap [0,1]$,
\begin{equation*}  
\lim_{n \to \infty} \EE\big[(\al^{(n)}_t)^2 \big] = \bar \EE [\al_t^2].
\end{equation*}
Since the paths $t \mapsto \EE\big[(\al^{(n)}_t)^2 \big]$ and $t \mapsto \bar \EE[\al_t^2]$ are non-decreasing, one can extend this convergence to all points of continuity of the latter mapping, which form a subset of $[0,1]$ of full measure. Using also \eqref{e.continuity.G}, we deduce that 
\begin{equation*}  
\lim_{n \to \infty} \inf_{t \in [0,1]} \int_t^1 \xi''(s) \big(\EE\big[(\al^{(n)}_s)^2 \big] - s\big) \, \d s = \inf_{t \in [0,1]} \int_t^1 \xi''(s) (\bar \EE[\al_s^2 ] - s) \, \d s.
\end{equation*}
In conclusion, we have constructed a martingale $\al$ over $\bar {\msc P}$ such that
\begin{equation*}  
f = \bar \EE \Ll[  \al_1 \int_0^1 \sqrt{\xi''(t)}\, \d W_t - \phi^*(\al_1)\Rr]  
 + \frac 1 2 \inf_{t \in [0,1]} \int_t^1 \xi''(s) (\bar \EE[\al_s^2 ] - s) \, \d s .
\end{equation*}	
Since the converse inequality \eqref{e.first.bound} is valid on every probability space, this completes the proof.
\end{proof}
We are now ready to prove the main results of this paper.

\begin{proof}[Proof of Theorems~\ref{t.main} and \ref{t.minimax}]
Let $\msc P$ be a probability space such that the conclusion of Lemma~\ref{l.max.mart} is valid. We denote by $\bar \al \in \bmart$ a martingale that achieves the supremum in \eqref{e.main}, and we denote by $\bar \mu \in \Pr([0,1])$ the probability measure that achieves the infimum in \eqref{e.parisi.formula}. We have already observed in \eqref{e.f.inf.sup} and in the paragraph below it that by \cite{auffinger2015parisi}, there is exactly one such choice of $\bar \mu$, and
\begin{equation*}  
f = \inf_{\mu \in \Pr([0,1])} \sup_{\al \in \bmart} \Gamma(\mu,\al) = \sup_{\al \in \bmart} \Gamma(\bar \mu,\al),
\end{equation*}
while Lemma~\ref{l.max.mart} and the integration by parts \eqref{e.ibp} state that 
\begin{equation*}  
f = \sup_{\al \in \bmart} \inf_{\mu \in \Pr([0,1])} \Gamma(\mu,\al) = \inf_{\mu \in \Pr([0,1])} \Gamma(\mu,\bar \al) .
\end{equation*}
In particular, for every $\mu \in \Pr([0,1])$ and $\al \in \bmart$, we have
\begin{equation*}  
\Gamma(\bar \mu,\al) \le f \quad \text{ and } \quad f \le \Gamma(\mu, \bar \al),
\end{equation*}
and thus $f = \Gamma(\bar \mu, \bar \al)$ and
\begin{equation}  
\label{e.pair.optimality}
\Gamma(\bar \mu,\al) \le \Gamma(\bar \mu, \bar \al) \le \Gamma(\mu, \bar \al).
\end{equation}
The first optimality condition in \eqref{e.pair.optimality} and Lemma~\ref{l.prog.to.mart} imply that $\bar \al$ must be as described in the statement of Theorem~\ref{t.minimax}. Notice that this martingale is measurable with respect to the filtration generated by the Brownian motion~$W$. In other words, there exists $\hat \al \in \mart$ such that the martingale $\bar \al \in \bmart$ is the image of $\hat \al$ under the canonical injection $\alpha \mapsto \alpha \circ W$ from $\mart$ to $\bmart$ (recall that $\mart$ denotes the space of bounded martingales over the Wiener space $\msc W$). The martingale $\hat \al$ has a canonical image in the space of bounded martingales of every probability space $\msc P$ under consideration in Theorem~\ref{t.main}. Recalling also from \eqref{e.first.bound} that the converse inequality is valid in every probability space, this completes the proof of Theorem~\ref{t.main}. The fact that the support of $\bar \mu$ is a subset of the set of maximizers of the mapping in~\eqref{e.def.map.t.min} is a consequence of the second optimality condition in \eqref{e.pair.optimality} and the integration by parts in~\eqref{e.ibp}. To show that the properties (1) and (2) in the statement of Theorem~\ref{t.minimax} characterize the pair $(\bar \mu, \bar \al)$ uniquely, we observe that these conditions imply that \eqref{e.pair.optimality} is valid for very $\mu \in \Pr([0,1])$ and $\al \in \bmart$, using again the integration by parts \eqref{e.ibp} and Lemma~\ref{l.prog.to.mart}, so we must have
\begin{equation*}  
f = \inf_{\mu \in \Pr([0,1])} \sup_{\al \in \bmart} \Gamma(\mu,\al) \ge \inf_{\mu \in \Pr([0,1])} \Gamma(\mu,\bar \al) = \Gamma(\bar \mu, \bar \alpha)
\end{equation*}
as well as
\begin{equation*}  
f = \sup_{\al \in \bmart} \inf_{\mu \in \Pr([0,1])} \Gamma(\mu,\al) \le \sup_{\al \in \bmart} \Gamma(\bar \mu,\al)  = \Gamma(\bar \mu,\bar \alpha). 
\end{equation*}
Hence $f = \Gamma(\bar \mu,\bar \alpha)$, the probability measure $\bar \mu$ must be the unique minimizer of \eqref{e.parisi.formula}, and the martingale $\bar \al$ must be the unique maximizer of \eqref{e.main}.
\end{proof}
\begin{proof}[Proof of Theorem~\ref{t.rsb}]
Let $\mu \in \Pr([0,1])$, let $(X_t)_{t \ge 0}$ be the strong solution to~\eqref{e.AC.SDE}, let $\alpha \in \bmart$ be such that \eqref{e.identity.alpha} holds for every $t \in [0,1]$, and let $g_\mu$ be such that, for every $t \in [0,1]$,
\begin{equation*}  
g_\mu(t) := \int_t^1 \xi''(s) (\EE[\al_s^2 ] - s ) \, \d s .
\end{equation*}
By Lemma~\ref{l.prog.to.mart}, we have that 
\begin{equation*}  
\Phi_\mu(0,h) = \EE \Ll[ \frac{1}{2}\int_{0}^{1}\xi''(t)\mu[0,t] \alpha^{2}_t\, \d t + \al_1 \int_0^1 \sqrt{\xi''(t)}\, \d W_t - \phi^*(\al_1)\Rr].
\end{equation*}
Combining this with Theorem~\ref{t.main} yields that
\begin{equation*}  
f \ge  \Phi_\mu(0,h) - \frac 1 2 \int_{0}^{1}\xi''(t)\mu[0,t] \EE[\alpha^{2}_t]\, \d t 
\\+ \frac 1 2 \inf_{\Ll[0,1\Rr]} g_\mu.
\end{equation*}
The integration by parts \eqref{e.ibp} therefore yields the second inequality in~\eqref{e.rsb}; we recall that the first one follows from \eqref{e.parisi.formula}. By Theorem~\ref{t.minimax}, if we choose~$\mu$ to be the minimizer of \eqref{e.parisi.formula}, then the support of $\mu$ is a subset of the set of minimizers of $g_\mu$, so the right-hand side of \eqref{e.rsb} vanishes in this case. Conversely, if $\mu \in \Pr([0,1])$ is such that the right-hand side of \eqref{e.rsb} vanishes, then it must clearly be the unique minimizer of \eqref{e.parisi.formula}.
\end{proof}

\small
\bibliographystyle{plain}
\bibliography{uninverting}

\end{document}